\documentclass[11pt]{article} 
\usepackage{authblk}
\usepackage{hyperref}
\usepackage{appendix}
\usepackage[extra]{tipa}
\newcommand*{\email}[1]{%
    \normalsize\href{mailto:#1}{#1}\par
    }
    
\usepackage[utf8]{inputenc} 
\usepackage{amssymb}
\usepackage{amsfonts}
\usepackage{amsmath}
\usepackage{mathtools}
\usepackage{xcolor}

\usepackage{tabu}

\usepackage{geometry} 
\usepackage{graphicx} 

\usepackage{booktabs} 
\usepackage{array} 
\usepackage{paralist} 
\usepackage{verbatim} 
\usepackage{subfig} 

\usepackage{fancyhdr} 
\usepackage{sectsty}
\allsectionsfont{\sffamily\mdseries\upshape} 

\usepackage[nottoc,notlof,notlot]{tocbibind} 
\usepackage[titles,subfigure]{tocloft} 

\newtheorem{theorem}{Theorem}

\newtheorem{lemma}[theorem]{Lemma}

\newtheorem{proposition}[theorem]{Proposition}
\newtheorem{remark}[theorem]{Remark}

\newcommand{\R}{\mathbb{R}}
\newcommand{\Z}{\mathbb{Z}}

\newcommand{\C}{\mathbb{C}}
\newcommand{\Id}{\mathrm{Id}}
\newenvironment{proof}[1][Proof]{\noindent\textbf{#1.} }{\ \rule{0.5em}{0.5em}}

\begin{document}
\title{Obstruction theory for the $\Z_2$-index  of $4$-manifolds }

\author{Chahrazade Matmat}
\affil{\small{   Département de mathématiques,
Université Constantine 1, Fr\`eres Mentouri, BP 325, Route Ain El Bey, 25017, Algeria.}

\email{c.matmat@yahoo.fr}}

\author{Christian Blanchet}
\affil{\small{ Universit\'e Paris Cit\'e, IMJ-PRG, UMR 7586 CNRS,  F-75013, Paris, France.}
\email{Christian.Blanchet@imj-prg.fr}}

\date{} 
\maketitle
\abstract{
We develop
a complete obstruction theory for the $\Z_2$-index of a compact connected 4-dimensional manifold with free involution. This $\Z_2$-index, equal to
the minimum integer $n$ for which there exists an equivariant map with target the $n$-sphere with antipodal involution, is computed in two steps using cohomology with twisted coefficients. The key ingredient
is a spectral sequence computing twisted cohomology of the orbit space of a free involution on odd complex projective spaces.
We illustrate the main results with various examples including computation of the secondary obstruction.}\\
\noindent
\textbf{2010 MSC} : 57M27, 57M60.\\
\textbf{Key words}:  4-manifolds, obstruction theory, spectral sequences, the Borsuk-Ulam theorem.
\section{Introduction}
For pairs $(X, \tau)$ with  $X $ a Hausdorff compact space and $\tau$ a fixed-point free involution on $X$, the $\Z_2$-index was  defined by Chung-Tao Yang in 1954 in \cite{Tao}. It has been mentioned multiple times in the literature; for instance, it is thought to be a useful tool for topological combinatorial applications. In this paper, we have an interest in one of the characterizations of the $\Z_2-$index. It says that it is the greatest value of $n$ satisfying that the pair $(X, \tau)$ has the Borsuk-Ulam property with respect to maps $f$ into $\R^n$. This means that there exists an $x\in X$ such that $f(\tau(x))=f(x)$.
The $\Z_2$-index also coincides with the minimum integer $n$ for which there exists an equivariant map with target the $n$-sphere with antipodal involution.\\
The investigation of the $\Z_2$-index of pairs $(X, \tau)$  is of a great historical interest in the domain of algebraic topology. The first pair referred to K. Borsuk and S. Ulam in 1933 \cite{Borsuk} was $(S^n, \tau)$, where $\tau$ is the antipodal map. Out of all the generalizations that may be made, we are most interested in the situation when $X$ is a low dimensional manifold. 
Numerous studies have addressed this case. The issue in dimension 2 was solved by Daciberg Lima Gonçalves in \cite{Daciberg2}. For dimension 3, certain classes of three dimensional manifolds have been studied in various papers. For examples, the case of double covering of Seifert manifolds was treated in \cite{BGHZ}, and that of $\mathbb{S}ol^3$-manifolds was covered in \cite{Hillman}. In \cite{blanchet}, we have discussed the general case of 3-manifolds. We have obtained criteria which allow us to give direct results for oriented 3-manifolds. Some interesting applications were given in this last paper, namely the Borsuk-Ulam theorem for double covers of torus bundles was studied.

In this paper, we study the Borsuk-Ulam property for pairs $(X, \tau)$ with respect to maps into $\R^n$ in the case where $X$ is a compact connected 4-dimensional manifold. In an equivalent manner, we investigate the $\Z_2$-index  which takes in our situation a positive value less than or equal to 4. Using obstruction theory and spectral sequences, we give criteria to distinguish between index 2 and 3. 

We begin with some notation. 
A free ${\Z}_2 $-space $(X, \tau)$ is a topological space $X$, which we assume to be path connected, and a fixed-point-free involutive homeomorphism $\tau$ on $X$. We denote by $N=X/\tau$ the orbit space.
Let  $x\in H^1(N, \Z_2)$ be the classifying class of  the principal $\Z_2$-bundle $X\twoheadrightarrow N$ i.e $x=\gamma^*(\alpha)$, where $\gamma:N\longrightarrow \R P^\infty$  classifies the bundle  and $\alpha$ is the generator of $H^1(\R P^\infty, \Z_2) \simeq \Z_2$. The class $x$ is not trivial since $X$ is path connected.

 For $n\geq 0$, the $n$-dimensional sphere $S^n$ is a $\Z_2$-space with the antipodal map $-\Id$. 
Let $(X, \tau)$ be a free $\Z_2$-space. An equivariant map $f$ between $(X, \tau) $ and $(S^n, -\Id)$  is a map $f: X \longrightarrow S^n$ wich commutes with the actions. In this case we write $f : X\overset{\Z_2}\longrightarrow S^n$. The $\Z_2$-index is defined as $$ind_{\Z_2}(X, \tau)=min\{n\in\{0, 1, 2,....\infty\}/ \exists f : X\overset{\Z_2}\longrightarrow S^n\}.$$ 
 Denote by $i_k$, $k \geq 1$, the inclusions  $\R P^k\longrightarrow \R P^\infty$, an equivalent form of the  definition above is the following 
$$ind_{\Z_2}(X, \tau)=min\{k\in\{0, 1, 2,....\infty\}/ \exists  \gamma_k : N \longrightarrow \R P^k , \gamma=i_k\circ \gamma_k\}.$$
 The $\Z_2$-index controls the Borsuk-Ulam property for pairs $(X, \tau)$ with respect to maps into $\R^n$, $n\geq 1$. The case where $X$ is a 
 connected $m$-dimensional CW-complex was treated in \cite{Daciberg3}.  From Lemma 2.4 and Theorems  3.1, 3.4 of this paper, we can extract the following assertions.
 We  denote by $\beta : H^*( \ .\ , \Z_2) \longrightarrow H^{*+1}( \ .\ , \Z),$ 
   the Bockstein homomorphism associated 
 to the short exact sequence $$0 \longrightarrow \Z \overset{\times 2}\longrightarrow \Z \longrightarrow \Z_2 \longrightarrow 0\ .$$ 
\begin{enumerate}
\item $1\leq ind_{\Z_2}(X, \tau)\leq m$.
\item $ind_{\Z_2}(X, \tau)=1 \Longleftrightarrow \beta(x)=0 $.
\item If $X$ is a compact connected $m$-dimensional manifold, then   \mbox{$ind_{\Z_2}(X, \tau)= m \Longleftrightarrow  x^m\neq 0$}.
\end{enumerate}
As a warming up for our methods, we give  a short proof using obsruction theory at the beginning of Section \ref{proof}. 
 
In the case of a $4$-dimensional manifold the above statements control when the $\Z_2$-index is $1$ or $4$. 
It remains to find a criterion for deciding between values $2$ and $3$. This is the purpose of this paper. The question will be solved by using classical tools, namely obstruction theory,
 in a case where we identify a primary obstruction in degree $3$ and possibly a secondary obstruction in degree $4$. 
The Borel construction in the case of a free $\Z_2$-space $(X,\tau)$ simply build a fibration
$$X\hookrightarrow X_{\Z_2}=\frac{X\times S^\infty}{\Z_2}\rightarrow \R P^\infty\ .$$
Here the cyclic group $\Z_2$ acts diagonally  on $X\times S^\infty$, and the fibration is defined from the second projection.  Using the first projection, we also have a fibration 
$$S^\infty\hookrightarrow X_{\Z_2}\rightarrow N=X/\tau\ ,$$
which shows that $X_{\Z_2}$ is homotopy equivalent to the orbit space $N$.
This construction allows to replace the inclusion $i_k:\R P^k\hookrightarrow \R P^\infty$ by
a fibration which 
will be also denoted by $i_k$. Then there is an obstruction theory for the lifting problem which can be formulated by using a Moore-Postnikov tower. 
 Denote by $E^{[2]}$ the first stage in the Moore-Postnikov tower associated to the fibration $i_2$.
 Then the primary obstruction is defined by the lifting problem of the classifying map $\gamma$ to $E^{[2]}$.
 We will provide a model for $E^{[2]}$ by using free involutions on odd dimensional complex projective spaces.
It appears that the vanishing of the primary obstruction in dimension $4$ is related with the
existence of an equivariant map from $(X,\tau)$ to $(\C P^3,\tau_3),$ where
$\tau_3([x_0,x_1,x_{2},x_{3}])=[-\overline x_1,\overline x_0,-\overline x_{3},\overline x_{4}]\ .$
We denote by  $q: \C P^3/\tau_3 \longrightarrow \R P^\infty$  a classifying map of the $\Z_2$-bundle $\C P^3\rightarrow \C P^3/\tau_3$.
We use the notation $\Z^-$ in the homology/cohomology of the orbit space of an involution for the local coefficient group $\Z$, where the action of the involution on chains is by $-1$.
 We establish the following lemma. 
\begin{lemma}\label{lem}
 Let $X$ be a compact  connected  $4$-dimensional CW-complex  with a fixed point free involution $\tau$, the orbit space $N=X/\tau$ and classifying map  $\gamma:N\longrightarrow \R P^\infty$. 
 The following statements are equivalent :
\begin{enumerate}[(i)]
\item [i)] There exists a lift $\gamma_1$ of $\gamma$ to $E^{[2]}$.
\item [ii)] There exists a lift $\tilde\gamma$ of $\gamma$ to $\C P^3/\tau_3$, i.e.  $\gamma=q\circ \tilde \gamma$.
\item [iii)] There exists an equivariant map from $(X,\tau)$ to $(\C P^3,\tau_3)$.
\end {enumerate}
\end{lemma}
{Now we are able to state the following theorem.
\begin{theorem}\label{th index 3} Let $X$ be a compact  connected  $4$-dimensional CW-complex  with a fixed point free involution $\tau$, the orbit space $N=X/\tau$, and the classifying class $x\in H^1(N,\Z_2)$. 
Let   $${\beta}^-: H^2 (N, \Z_2) \longrightarrow H^3(N, \Z^-)$$
be the Bockstein homomorphism associated to the sequence of coefficients
 $$0 \longrightarrow \Z^- \overset{\times 2}\longrightarrow \Z^- \longrightarrow \Z_2 \longrightarrow 0.$$
 \begin{itemize}
\item [a)] If $ind_{\Z_2}(X, \tau)\leq 2$ then ${\beta}^-(x^2) =0.$ (Here $x^2=x\smile x$ is the square cup product.) 
\item[b)] The statements (i), (ii) , (iii) in Lemma \ref{lem} are equivalent to 
\begin{enumerate}[(iv)]
\item [(iv)] The class
${\beta}^-(x^2)$ vanishes. 
\end{enumerate}
\end{itemize}
 \end{theorem}
A lift $\tilde \gamma: N\longrightarrow \C P^3/\tau_3$ of the classifying map $\gamma: N\longrightarrow \R P^\infty$ is covered by an equivariant map $\hat \gamma: (X,\tau)\longrightarrow (\C P^3,\tau_3)$ determined up to composition with the involution $\tau$, meaning that the other equivariant map over $\tilde \gamma$ is $\hat \gamma\circ \tau=\tau_3\circ \hat \gamma$.
   In the case where the primary obstruction ${\beta}^-(x^2)$ vanishes, the secondary obstruction contains the information about reducing a lift $\tilde \gamma: N\longrightarrow \C P^3/\tau_3$ to $\R P^2$ or equivalently 
  reducing an equivariant map $\hat{\gamma}: (X,\tau)\longrightarrow (\C P^3,\tau_3)$ to $S^2=\C P^1$ (with antipodal involution).
  \begin{lemma}\label{CP3}
  The covering map $\C P^3 \longrightarrow \C P^3/\tau_3$ induces an isomorphism
$$H^4(\C P^3/\tau_3,\Z)\cong H^4(\C P^3,\Z)\cong \Z\ .$$
  \end{lemma}  
  We denote by $C\in H^4(\C P^3/\tau_3,\Z)$
the class corresponding to the square of  the generator $c^2\in H^4(\C P^3,\Z)$. 
  \begin{theorem}~ \label{Th2}
 With the above notation, in the case where $\beta(x)\neq 0$ and ${\beta}^-(x^2) = 0$ we have :
\begin{itemize}
\item[a)]   A lift of $\gamma$,  $\tilde \gamma: N\longrightarrow \C P^3/\tau_3$ reduces to $\R P^2$  if and only if $\tilde\gamma^*(C)\in H^4(N,\Z)$ vanishes.
\item[b)] The index  $ind_{\Z_2}(X,\tau)$ is equal to $2$ if and only if there exists a lift of $\gamma$, \mbox{$\tilde\gamma: N\longrightarrow \C P^3/\tau_3$} such that $\tilde\gamma^*(C)=0$.
\end{itemize}
  \end{theorem}
  
 We will discuss the computation of the secondary obstruction in the case where 
$\tau$ is oriented and $\beta(x)\neq 0$ (i.e. $ind_{\Z_2}(X, \tau)>1$).
If $\beta^-(x^2)=0$, then the homotopy classes of lifts $\tilde \gamma: N\longrightarrow \C P^3/\tau_3$ are parametrised by $H^2(N,\Z^-)$. We describe in Appendix the Leray-Serre spectral sequence $(\overline E_*^{p,q}, \overline d_*^{p,q})$ converging to $H^*(N,\Z^-)$. In the next theorem we use this spectral sequence to achieve the computation of the secondary obstruction.
We denote by $H^n(X,\Z)^{as}\subset H^n(X,\Z) $  the subgroup of anti-symmetric classes $y$, $\tau^*(y)=-y$. 
\begin{theorem}\label{th5} 
 Let $X$ be a compact  connected  oriented $4$-dimensional manifold  with a fixed point free oriented involution $\tau$, the orbit space $N=X/\tau$, the classifying map  $\gamma:N\longrightarrow \R P^\infty$ and the classifying class $x\in H^1(N,\Z_2)$. 
  Suppose that $\beta (x)$ is non zero,  then we have
  \begin{enumerate}
\item The primary obstruction $\beta^-(x^2)$ vanishes if and only if the differential $\overline d_3^{0,2}$ is surjective.
\item 
$ind_{\Z_2}(X, \tau)= 3$ if and only if 
$$\forall x\in H^2(X,\Z)^{as}; \; \bar{d}_3^{0,2}(x)\neq 0\Rightarrow x^2\neq 0\in H^4(X,\Z).$$
\end{enumerate}
\end{theorem}
We complete the results with the following proposition, which we will prove at the end of section 3.
\begin{proposition}\label{prop6}
In the case where $\tau$ is an oriented involution on $X$ ($N=X/\tau$ is an oriented manifold), we have $ind_{\Z_2}(X, \tau)<4.$ 
\end{proposition}

This paper is divided into four sections and an appendix.  The main results stated in Section 1 will be proved  in Section 3. In Section 2, we will study the cohomology with twisted coefficients for the orbit space of a free involution acting on odd complex projective spaces. In Section 4, we will have some families of  examples and applications. All free involutions on $S^1 \times S^3$ will be studied. Following that, we will consider an oriented example using presentation of closed oriented 4-dimensional manifolds. Then we will consider a case when the primary obstruction vanishes. This enables us to investigate the secondary obstruction. We end this paper with an appendix about the Leray-Serre spectral sequence  over $\Z_2$-twisted coefficients.

\section{Free involution on odd dimensional complex projective spaces}
\label{free_invol}
Free involutions on odd complex projective spaces are studied in \cite{Singh}.
For $n\geq 0$, the map $\tau_n:\C P^{2n+1}\rightarrow \C P^{2n+1}$  defined by $\tau_n([x_0,x_1,\dots,x_{2n},x_{2n+1}])=[-\overline x_1,\overline x_0,\dots,-\overline x_{2n+1},\overline x_{2n}]\ $
is a free involution. We denote by $\tau_\infty:\C P^{\infty}\longrightarrow \C P^{\infty}$ its inductive limit.
  The orbit space $\C P^{\infty}/\tau_\infty$ is homotopy equivalent to  the Borel construction $E^{[2]}=(\C P^\infty)_{\Z_2}=\frac{\C P^\infty \times S^\infty}{\Z_2}$ 
and we have a fibration
$\C P^\infty\hookrightarrow E^{[2]} \overset{p_2}\longrightarrow \R P^\infty\ .$
 The fibration $p_2$ whose fiber is $\C P^\infty= K(\Z,2)$ is
 the first stage in the Moore-Postnikov tower associated with the inclusion $i_2: \R P^2\rightarrow \R P^\infty$. 
 In order to prove Theorems  \ref{th index 3} and \ref{Th2}, we will need to understand
  \begin{enumerate}
  \item the first obstruction for lifting  to $E^{[2]}$ a map $f:N\longrightarrow \R P^\infty$, which lives
  in $H^3(\R P^\infty,\pi_2(\C P^\infty))$,
  \item the cohomology $H^4(E^{[2]}, \Z )$, which contains  the secondary obstruction  defined from the lifting problem of maps  $N \longrightarrow E^{[2]}$ to $\R P^2$. 
  \end{enumerate}
    We observe that the action of $\pi_1(\R P^\infty)$ on $\pi_2(\C P^\infty)$ is non trivial, so that we have to deal with local coefficients.
  Recall that we denote by $H^*(\R P^\infty,\Z^-)$ the cohomology of $\R P^\infty$ with local coefficients given by the infinite cyclic group with non trivial action of $\pi_1(\R P^\infty)=\Z_2$. We compute below $H^*(\R P^\infty,A)$ for any $\Z[\Z_2]$-module $A$.
  \vspace{10pt}\begin{proposition}\label{prop5}
 \label{proj_s_as}Let $A$ be a $\Z[\Z_2]$-module. 
If we denote by $A^s\subset A$ (resp. $A^{as}\subset A$) the subgroup of symmetric (resp. anti-symmetric) classes under the action $\tau$ of the generator of $\Z_2$, then 
$$H^p(\R P^\infty,A)=\left\{\begin{array}{l}A^s \text{ if $p=0$},\\
A^{as}\otimes \Z_2 \text{ if $p$ is odd},\\
A^{s}\otimes \Z_2 \text{ if $p$ is  even positive}.
\end{array}
\right.$$
In particular
$H^n(\R P^\infty,\Z^-)=\left\{\begin{array}{l}
  0\text{ if n is even}\\
  \Z_2\text{ if n is odd}
  \end{array}\right.$
\end{proposition}
  
  \begin{proof} 
  We use an equivariant cell decomposition of the universal cover $S^\infty$ with two cells in each dimension $k$, namely $e_k, Te_k$ with $T$ the antipodal map.  
As a $\Z[Z_2]$-module, the cellular complex 
$C_*(S^\infty)$ has generator $e_k$ in dimension $k$ and boundary map $\partial e_k=e_{k-1} + (-1)^{k}Te_{k-1}$.
Here  $(-1)^k$ is the orientation sign coming from the action of $T$.
 The cohomology complex $C^*(\R P^\infty,A)=Hom_{\Z[\Z_2]}(C_*(S^\infty),A)$ is
$$A\overset{1-\tau}\longrightarrow A\overset{1+\tau}\longrightarrow A\overset{1-\tau}\longrightarrow A\overset{1+\tau}\longrightarrow A\dots$$
It follows that 
\begin{align*}
H^0(\R P^\infty,A)&=\ker(1-\tau)=A^s,\\
  H^p(\R P^\infty,A)&=\ker(1+\tau)/\mathrm{Im}(1-\tau)=A^{as}/2A^{as},
 \text{ if $p$ is odd},\\
 H^p(\R P^\infty,A)&=\ker(1-\tau)/\mathrm{Im}(1+\tau)=A^{s}/2A^{s}, 
\text{ if $p$ is even positive}.
 \end{align*}
  The statement follows.
 \end{proof}
The cohomology rings $H^*(\C P^n/\Z_2,\Z_2)$, $n$ odd, are computed in \cite{Singh} from the modulo $2$ Leray-Serre spectral sequence.
A presentation is given by $H^*(\C P^n/\Z_2,\Z_2)=\Z_2[x,y]/<x^3,y^{(n+1)/2}>$,
where $x$ is in degree one and $y$ is in degree $2$.
It follows that $H^*(E^{[2]},\Z_2)=H^*(\C P^\infty/\Z_2,\Z_2)=\Z_2[x,y]/<x^3>$. Here we complete the computation for integer coefficients and twisted coefficients $\Z^-$.
\vspace{10pt}\begin{theorem}~\label{cohoE2} We have
 $$H^k(E^{[2]},\Z)=\left\{
\begin{array}{l}
0\text{ if $k$ is odd,}\\
\Z\text{ if $k\equiv 0$ mod $4$,}\\
\Z_2\text{ if $k\equiv 2$ mod $4$.} 
\end{array}\right.\hfill\hspace{1cm}
 H^k(E^{[2]},\Z^-)=\left\{
\begin{array}{l}
\Z_2\text{ if $k\equiv 1$ mod $4$,}\\
\Z\text{ if $k\equiv 2$ mod $4$,} \\
0\text{ else.}\\
\end{array}\right.$$
\end{theorem}
\begin{proof}
 The projection $p_2: E^{[2] } \longrightarrow \R P^\infty$ is a fibration with fiber $\C P^\infty$. Here
 $\pi_1(\R P^\infty)=\Z_2$ acts non trivially on $H^*(\C P^\infty)=\Z[c]$, so that the Leray-Serre spectral sequence involves local coefficients. \\We get 
 a spectral sequence $\{E_r^{p,q}, d_r\},$ where  \mbox{$E_2^{p,q}= H^p( \R P^\infty, H^q (\C P^\infty, \Z))$} and $d_r^{p,q} : E_r^{p,q} \longrightarrow E_r^{p+r,q-r+1}$, whose limit is $H^{p+q}(E^{[2]}, \Z)$.
The construction of this spectral sequence involves the filtration coming from a CW-complex structure of the basis $\R P^\infty$. This can be adapted to the twisted coefficient group $\Z^-$. For convenience we give the useful formulation of the Leray-Serre spectral sequence in the Appendix. From Theorem \ref{LSSS} we have a spectral sequence $\{\overline E_r^{p,q}, \overline d_r\},$ where  \mbox{$\overline E_2^{p,q}= H^p( \R P^\infty,  H^q (\C P^\infty, \Z)\otimes \Z^-)$} and $\overline d_r^{p,q}: \overline E_r^{p,q} \longrightarrow \overline E_r^{p+r,q-r+1}$, whose limit is $H^{p+q}(E^{[2]},\Z^-)$. Here the local groups 
$H^q (\C P^\infty, \Z)$ and $H^q (\C P^\infty, \Z)\otimes \Z^-$ are the $\Z[\Z_2]$-modules given by 
$$H^q(\C P^\infty,\Z)=\left\{\begin{array}{ll}
\Z^-&\text {if } n\equiv 2 \text{ mod } 4,\\
\Z&\text {if } n\equiv 0 \text{ mod } 4,\\
0& \text{ else. }\end{array}\right.\ \ H^q(\C P^\infty,\Z)\otimes \Z^-=\left\{\begin{array}{ll}
\Z&\text {if } n\equiv 2 \text{ mod } 4,\\
\Z^-&\text {if } n\equiv 0 \text{ mod } 4,\\
0& \text{ else. }\end{array}\right.$$
We write below the bottom-left part of the  second pages $E_2$ and $\overline E_2$.
$$ 
 \begin{tabular}{|c|c|c|c|c|}
\hline 
$\Z $& 0  & $\Z_2$ & 0 & $\Z_2$ \\ 
\hline 
0 & 0 & 0 & 0 & 0 \\ 
\hline 
0 & $\Z_2$ & 0 &$ \Z_2$& 0 \\ 
\hline 
0 & 0 & 0 & 0 & 0 \\ 
\hline 
$E_2^{0,0} \simeq \Z$ & 0 & $\Z_2$& 0& $\Z_2$\\ 
\hline 
\end{tabular} 
 \hspace{2cm}
 \begin{tabular}{|c|c|c|c|c|}
\hline 
 0  & $\Z_2$ & 0 & $\Z_2$ &0\\ 
\hline 
0 & 0 & 0 & 0 & 0 \\ 
\hline 
 $\Z$ & 0 &$ \Z_2$& 0&$\Z_2$ \\ 
\hline 
0 & 0 & 0 & 0 & 0 \\ 
\hline 
$\overline E_2^{0,0} \simeq 0$  & $\Z_2$& 0& $\Z_2$&0\\ 
\hline 
\end{tabular} $$
We quote that in the spectral sequences the coefficients are local. The groups $ E_2^{p,q}$ (resp. $ \overline E_2^{p,q}$) vanish for odd $q$, so that $d_2$ (resp. $\overline d_2$) is zero and $E_3^{p,q} \simeq E_2^{p,q}$ (resp. $\overline E_3^{p,q} \simeq \overline E_2^{p,q}$)
for all $p, q$.\\
We first study $\overline E_3$ which gives $H^0(E^{[2]},\Z^-)=0$, $H^1(E^{[2]},\Z^-)=\Z_2$, \mbox{$H^2(E^{[2]},\Z^-)=\ker(\overline d_3^{0,2})$}
is a free group, $H^3(E^{[2]},\Z^-)=\mathrm{coker}(\overline d_3^{0,2})$. 
We claim that $\overline d_3^{0,2}$ is non zero. This can be seen as follows. The space $E^{[2]}$ was constructed by applying the well known Borel construction to $\C P^\infty$ with its free involution.  This Borel construction  can be restricted to the subspace 
$\C P^1=S^2$ and produces a space which is homotopy equivalent to $\R P^2$. For this subspace, the second page of the spectral sequence coincides with the three bottom rows in $\overline E_2$ and converges to $H^*(\R P^2,\Z^-).$ The vanishing of $H^3(\R P^2,\Z^-)$ forces the differential $\overline d_3^{0,2}$ to be non zero. Our argument also proves that $\overline d_3^{2k,2}$ is non zero for $k>0$.  
 We will use below the product structure, see  Remark \ref{product} in Appendix.
Let us denote by $\alpha_-$ the generator of $H^1(\R P^\infty,\Z^-)$, and $c$  the generator of $H^2(\C P^\infty,\Z)$. The non zero groups in $\overline E_3$ are $\overline E^{p,2k}$ where $p\geq 0$ and $k\geq 0$ have different parity, generated  by the products $(\alpha_-)^p c^k$. The resulting differential $\overline d_3$ is $\overline d_3^{2\nu+1,4l}=0$ for $\nu,l\geq 0$,  $\overline d_3^{0,4l+2}$ is surjective for $l\geq 0$,
and $\overline d_3^{2\nu,4l+2}$ is an isomorphism for $\nu>0,l\geq 0$. 
The  page $\overline E_4$ gives the limit which contains the expected groups.
We now consider the spectral sequence $E$. 
From $E_3$,
we obtain $H^0(E^{[2]},\Z)=\Z$, $H^1(E^{[2]},\Z)=0$, $H^2(E^{[2]},\Z)=\Z_2$, $H^3(E^{[2]},\Z)=\ker(d_3^{1,2})\subset \Z_2$. Similarly to the previous computation,  a restriction argument shows that
 $d_3^{2k+1,2}$ is an isomorphism for $k\geq 0$.
We have a  product $\overline E_3\otimes \overline E_3 \rightarrow E_3$ which is compatible with the differential.
 We have that  $c^2$ is the generator of $H^4(\C P^\infty,\Z)={E}_3^{0,4}$. We obtain that $d_3(c^2)=2 c\overline d_3(c)=0$. Hence, $d_3^{0,4}$ is zero.
Using the product structure again we obtain  $ d_3^{2\nu,4l}=0$  and $ d_3^{2\nu+1,4l+2}$ is an isomorphism for  $\nu,l\geq 0$. The computation of all groups follows.
\end{proof}
\begin{remark}
Our computation shows that the generator of $H^2(E^{[2]},\Z^-)$ is $2c$ while the generator of $H^4(E^{[2]},\Z)$ is $C=c^2$. So the square map
 $H^2(E^{[2]},\Z^-)\longrightarrow H^4(E^{[2]},\Z)$ sends the generator $c^-=2c$ to $4C$.
\end{remark}
\section{Proof of the main theorems}\label{proof}
In this section we will prove our main results. As announced in the introduction, we start with an obstruction theoretic proof of the known results \cite{Daciberg 3}.
\begin{proposition}Let $X$ be an $m$-dimensional CW-complex with a fixed point free involution $\tau$, the orbit space $N$ and the classifying class $x$.
\begin{itemize}
\item[a)] $1\leq ind_{\Z_2}(X, \tau)\leq m$.
\item[b)] $ind_{\Z_2}(X, \tau)=1 \Longleftrightarrow \beta(x)=0 $.
\item[c)] If $m$ is odd, then \mbox{$ind_{\Z_2}(X, \tau)= m \Longleftrightarrow  \beta x^{m-1}\neq 0$}.
\item[d)] If $m$ is positive even, then \mbox{$ind_{\Z_2}(X, \tau)= m \Longleftrightarrow  \beta^- x^{m-1}\neq 0$}.
\item[e)] If $X$ is a compact connected $m$-dimensional manifold, then \mbox{$ind_{\Z_2}(X, \tau)= m \Longleftrightarrow  x^m\neq 0$}.
\end{itemize}
\end{proposition}
\begin{proof}
The reduction problem of the classifying map $\gamma: N\rightarrow \R P^\infty$
 to $\R P^k$ can be replaced by the lifting problem for the fibration
$$S^k\hookrightarrow \frac{S^k\times S^\infty}{\Z_2}\rightarrow \R P^\infty.$$
The first obstruction lives in $H^{k+1}(N,\pi_k(S^k))$. Here $\pi_k(S^k)$ is a local group corresponding to the homological action of the antipodal map, $\pi_k(S^k)=\Z$ if $k$ is odd, and $\pi_k(S^k)$ is equal to the twisted 
cyclic group denoted by $\Z^-$ if $k$ is even.
 This first obstruction vanishes if $k$ is bigger than the dimension $m$ of the CW-complex $X$, which proves a).  
 This obstruction is given by a non trivial universal class in 
 $$H^{k+1}(\R P^\infty,\pi_k(S^k))=\left\{\begin{array}{l}
 H^{k+1}(\R P^\infty,\Z)=\Z_2 \text{ if $k$ is odd,}\\
  H^{k+1}(\R P^\infty,\Z^-)=\Z_2 \text{ if $k$ is even,}\end{array}\right. $$ 
  with generator  $\beta(\alpha^k)$ if $k$ is odd, and $\beta^-(\alpha^k)$ if $k$ is even.
 See Proposition \ref{proj_s_as} for the twisted case. For $k=1$, we get the unique obstruction class
 $\beta(x)$, which proves b). For $k=m-1$, we have a unique obstruction
class which, by functoriality, is  $\beta(x^{m-1})\in H^m(N,\Z)$ if $m$ is even, and  $\beta^-(x^{m-1})\in H^m(N,\Z^-)$ if $m$ is odd.
We have proved statements c) and d). 
 The reduction of the obstruction class to  $H^m(N,\Z_2)$, read from $H^m(\R P^\infty,\Z_2)$, is equal to $x^m$.
If $X$ is a compact connected $m$-dimensional manifold, then either the top cohomology group
 $H^m(N,\Z^\pm)$ is infinite cyclic or has order $2$, and in both cases the vanishing of the obstruction is equivalent to the vanishing of its modulo $2$ reduction $x^m\in H^m(N,\Z_2)$ which proves e).
 \end{proof}
 
\begin{proof}[Proof of Lemma \ref{lem}]
The inclusion $ \C P^3/\tau_3 \subset \C P^\infty/\Z_2$ induces isomorphisms in $\pi_m$ for $m\leq 4$, which implies that $\C P^3/\tau_3$ is  $4$-equivalent to $\C P^\infty/\Z_2$. That is canonically homotopy equivalent to $E^{[2]}$. This implies the equivalence between  (i) and (ii). \\
An equivariant map $h: (X,\tau)\rightarrow (\C P^3,\tau_3)$ induces a quotient map $\overline h: N=X/\tau\rightarrow \C P^3/\tau_3$ which is compatible with the classifying maps up to homotopy. This means that $q\circ \overline h: N\rightarrow \R P^\infty$ is a classifying map;  it is homotopic to $\gamma$ which hence has a lift to $\C P^\infty/\tau_3$. 
Conversely, a map
 $\tilde \gamma:N\rightarrow \C P^\infty/\tau_3$ such that $\gamma=q\circ \tilde \gamma$ satisfies $\tilde \gamma^*(x_3)=x$ where $x_3=q^*(\alpha)$ is the classifying class for $\C P^3\rightarrow \C P^3/\tau_3$.
 It follows that $\tilde \gamma$ can be lifted to the double coverings giving an equivariant map. This achieve the equivalence between (ii) and (iii). \end{proof}

\vspace{10pt}\begin {proof}[Proof of Theorem \ref{th index 3}]
 We have that $ind_{\Z_2}(X,\tau)\leq 2$ if and only if the classifying map $\gamma: N\rightarrow \R P^\infty$ has a lift $\gamma_2: N\rightarrow \R P^2$. If such lift exists, then, since $H^3(\R P^2,\Z^-)= 0$, we have by functoriality of the Bockstein map that $\beta^-(x^2)=0$, which proves the statement a).\\
  The first obstruction for lifting $\gamma : N\longrightarrow \R P^\infty$ to $\R P^2$ is the unique obstruction for lifting $\gamma$ to $E^{[2]}=\C P^\infty / \tau_\infty$ and lives in $H^3(N,\pi_2(\C P^\infty))=H^3(N,\Z^-)$. Here the action of $\pi_1$ is non trivial whence the use of local coefficients.
 This obstruction is a universal class $\gamma^{*}(o)$ with $o\in H^3(\R P^\infty,\Z^-)=\Z_2$ and it is non zero since $\R P^\infty$ does not retract on $\R P^2$. The unique non zero class in  $H^3(\R P^\infty,\Z^-)$ is $\beta^-(\alpha^2)$, where $\alpha$ is the generator of $H^1(\R P^\infty,\Z/2\Z)$. By functoriality we have that the obstruction is $\gamma^{*}(o)=\beta^-(x^2)$.
 We have $\beta^-(x^2) =0$ if and only if there exists a lift of $\gamma$, $\tilde\gamma : N \longrightarrow E^{[2]}\cong \C P^\infty/\tau_\infty$. Using  Lemma \ref{lem} we  obtain the  statement b).
\end{proof}
\vspace{0,5 cm}

 When the first obstruction $\beta^-(x^2)$ vanishes then the liftings up to homotopy are parametrized by $H^2(N,\Z^-)$.
 For each of them we have a secondary obstruction which we study now for proving  Theorem \ref{Th2}.
 
\vspace{10pt}\begin{proof}[Proof of Lemma \ref{CP3}]
The calculation in the spectral sequence for the fibration $i_2$ in the proof of Theorem \ref{cohoE2} proves that $H^4(\C P^\infty/\tau_\infty, \Z) = H^4(\C P^\infty, \Z)\simeq \Z$ generated by the square of the generator $c\in H^2(\C P^\infty)$. A similar computation can be done for computing $H^*(\C P^3/\tau_3, \Z)$, where the spectral sequence is truncated by the relation $c^4=0$. We obtain that $H^4(\C P^3/\tau_3, \Z) = H^4(\C P^3, \Z)\simeq \Z$ generated by the square $C=c^2$ of the generator $c\in H^2(\C P^3,\Z)$. \end{proof}

\vspace{10pt}\begin{proof}[Proof of Theorem \ref{Th2}]
 Since $H^4(\R P^2,\Z)=0$, the vanishing of $\tilde \gamma^*(C)$ is a necessary condition for existence of a lift
 to $\R P^2$ of a map $\tilde \gamma: N\rightarrow \C P^3/\tau_3$. We will show that the condition is sufficient by using basic obstruction theory. 
We have that $\pi_3(\R P^2)=\pi_3(S^2)=\Z$, generated by the Hopf map. Here the action of $\pi_1(\R P^2)$
is trivial. This can be seen as follows. The Hopf map $h: S^3\rightarrow \C P^1=S^2$ is defined by
$h(z,z')=[z,z']$. The deck transformation on $\C P^1$ is $\tau_1: [z,z']\mapsto [-\overline z',\overline z]$. We  have
that $\tau_1 \circ h=h\circ v$, where $v(z,z')=(-\overline z',\overline z)$ defines an oriented homeomorphism which is homotopic to the identity via an isotopy which fixes the image of the base point in $\C P^1$ say $h(1,0)$.\\
Let us denote by 
$\overline{i}_2: \R P^2\rightarrow \C P^3/\tau_3$ the quotient map of the inclusion $i_2: \C P^1\rightarrow \C P^3$.
This map can be replaced by a fibration whose fiber $F$, so called the homotopy fiber of $\overline{i}_2$, controls the lifting problem. From the homotopy exact sequence of the fibration, we obtain that the first non trivial homotopy group of $F$ is $\pi_3(F)=\pi_3(S^2)=\Z$.
The first obstruction for lifting a map $\tilde \gamma : N\rightarrow \C P^3/\tau_3$ to $\R P^2$ belongs to $H^4(N,\pi_3(F))=H^4(N,\Z)$.  By functoriality, this obstruction vanishes if $\tilde \gamma^*(C)=0$. 
Here $N$ has dimension $4$ and there is no further obstruction, which completes the proof of a).\\
Using Theorem \ref{th index 3}, and the statement a) above, we obtain the proof of the statement b).
 \end{proof}
 
\vspace{10pt}\begin{proof}[Proof of Theorem \ref{th5}]~~
 The Borel construction $X \times_{\Z_2} S^\infty$ is homotopy equivalent to the orbit space $N =X/\tau$. We use the Leray-Serre spectral sequences $E_r$, $\overline{E_r}$ for studying  the cohomologies $H^*(N,\Z)$,  $H^*(N,\Z^-)$, see  Theorem \ref{LSSS} and Remark \ref{proj_s_as} in the appendix. 
 Under our hypothesis $\beta (x)\neq 0$ we get that the differential \mbox{$d_2^{0,1}:E_2^{0,1}=H^1(X,\Z)^s\rightarrow E_2^{2,0}=\Z_2$} vanishes. Using the product map \mbox{$E_2^{0,1}\otimes \overline E_2^{1,0}=H^1(X,\Z)^s\otimes \Z_2\stackrel{\sim}{\rightarrow} \overline E_2^{1,1}$}, we get that the differential \mbox{$\overline d_2^{1,1}:\overline E_2^{1,1}=H^1(X,\Z)^s\otimes \Z_2\rightarrow  \overline E_2^{3,0}=\Z_2$} also vanishes.
 We now consider \mbox{$\bar{d}^{0,2}_3 : 
 H^2(X,\Z)^{as}  \rightarrow \Z_2=H^3(\R P^\infty, \Z^-)$}. 
 If we suppose that $\bar{d}^{0,2}_3$ is surjective, we obtain that $\beta^-(x^2),$ which by functoriality lives in $\overline E_\infty^{0,3}=0$, vanishes.
 Conversely, suppose that $\bar{d}^{0,2}_3 =0$. Using the functoriality of the Bockstein homomorphism, we get that $\beta^-(x^2)= \gamma^*(\beta^{-}(\alpha^2))\neq 0$, where $\alpha$ is the generator of $H^1(\R P^\infty, \Z_2)$ and  $\beta^{-}$ in the second member is $\beta^{-} : H^2(\R P^\infty, \Z_2) \rightarrow H^3(\R P^\infty, \Z^-)$.
This completes the proof of statement 1. 

In the case where the primary obstruction $\beta^-(x^2)$ vanishes,
 the page $\overline E_3$ gives the quotient groups of the skeleton filtration 
  of $H^2(N,\Z^-)$, which can be formulated as the following exact sequence
\begin{equation}\label{seq_d3}
 0\rightarrow \overline E_3^{1,1}=H^1(X,\Z)^s\otimes \Z_2\rightarrow H^2(N,\Z^-)\stackrel{p^*}{\rightarrow} \ker(\overline d_3^{0,2})
 \rightarrow 0\ \end{equation}
 The last homomorphism is associated with the covering map $p: X\rightarrow N$.
Let \mbox{$\tilde{\gamma} : N \rightarrow \C P^3/\tau_3$} be a lift of $\gamma$. Here we have fixed base points and are working with pointed maps. We have a unique based equivariant  lift $\hat \gamma:
 X\rightarrow \C P^3$. We get a cohomology class $\hat \gamma^*(c)$ where $c$ is the canonical generator of $H^2(\C P^3,\Z)$.
 By functoriality we have that this class is antisymmetric and  $\overline d_3^{0,2}(\hat \gamma^*(c))\neq 0$.
 This defines a map from the set $R(\gamma)$ of homotopy classes of lifts $\tilde \gamma$ to
  $$L(X)=\{s\in H^2(X,\Z)^{as}, \ \overline d_3^{0,2}( s )=1 \text{\ mod\ $2$}\}\ .$$
  The group  $H^2(X,\Z)^{as}$ splits as a disjoint union of $\ker \overline d_3^{0,2}$ and $L(X)$, and we have a simply transitive action of the subgroup $\ker \overline d_3^{0,2}$ on $L(X)$ (affine structure).
 Given two lifts $\tilde \gamma$, $\tilde{\tilde\gamma}$, obstruction theory \cite[VI.5]{Whitehead} defines a difference cocycle 
 $d(\tilde \gamma,\tilde{\tilde \gamma})$ with value in $\pi_2(\C P^3)$ 
 whose cohomology class gives an affine structure over $H^2(N,\Z^-)$.
 Using the projection homomorphism $p^*: H^2(N,\Z^-)\rightarrow \ker \overline d_2^{0,2}\subset H^2(X,\Z)^{as}$, we get a class 
 $p^*(d(\tilde \gamma,\tilde{\tilde \gamma}))\in \ker \overline d_2^{0,2}\subset H^2(X,\Z)$.
The evaluation of the difference cocycle $d(\tilde \gamma,\tilde{\tilde \gamma})$ on a $2$-cell $e\subset N$ is equal to the evaluation of 
 $p^*(d(\tilde \gamma,\tilde{ \tilde \gamma}))$ on a lift $\hat e\in X$.
 Using the isomorphism $[X,\C P^3]\cong H^2(X,\Z)$, we get that
 $p^*(d(\tilde \gamma,\tilde{\tilde \gamma}))$ is equal to the difference class given by the two equivariant lifts $\hat \gamma$, $\hat{\hat \gamma}$. 
 It follows that $p^*(d(\tilde \gamma,\tilde{\tilde \gamma}))=\hat{\hat \gamma}^*(c)-\hat \gamma^*(c)$.  
 We deduce that the image of the map $\tilde \gamma\mapsto \hat \gamma^*(c)$ is equal to $L(X)$

The covering map $p: X\rightarrow N$ being oriented, the map $p^*: \Z=H^4(N,\Z)\rightarrow  H^4(X,\Z)=\Z$ is multiplication by $2$, in particuler it is injective.  It follows that the secondary obstruction  $\tilde \gamma^*(C)$ vanishes if and only if $p^*(\tilde \gamma^*(C))=0$. 
Recall from Theorem \ref{th index 3} that $C\in H^4(\C P^3/\tau_3,\Z)$ corresponds to $c^2\in H^4(\C P^3,\Z)$. We deduce $p^*(\tilde \gamma^*(C))=\hat \gamma^*(c^2)$.

 From Theorem \ref{Th2} we know that in the hypothesis  $\beta^-(x^2)=0$, we have  $ind_{\Z_2}(X, \tau)<3$ if and only if $\tilde \gamma^*(C)=0$ for all lifts $\tilde \gamma: N\rightarrow \C P^3/\tau_3$. 
 We proved that this is equivalent to $\hat \gamma^*(c)^2=0$ for all $c\in L(X)$, which establishes the statement 2.
\end{proof}

\vspace{10pt}\begin{proof}[Proof of Proposition \ref{prop6}]
Here we will prove that we have always $ind_{\Z_2}(X,\tau) \neq 4$. We know that it is satisfied when $x^4=0$. This means that the classifying class $\gamma$ compress to $\R P^3$. By the same arguments in the proofs of Theorems \ref{th index 3}, \ref{Th2}, it is easy to check that $x^4= \rho\circ \beta (x^3)$, where $\rho$ is the homomorphism which reduces the coefficients mod 2, $\beta : H^3(N, \Z_2) \longrightarrow H^4(N, \Z)$ is the Bockstein homomorphism associated to the short exact sequence $0 \longrightarrow \Z \overset{\times 2}\longrightarrow \Z \longrightarrow \Z_2 \longrightarrow 0$. Since $\beta(x^3)$ always vanishes in $H^4(N, \Z) \simeq \Z$, we obtain the desired result.
\end{proof}

\section{Some examples and applications}\label{Examples}
\subsection{The Borsuk-Ulam theorem for $S^1\times S^3$}
From \cite{Jahren} we have that a free involution on $S^1\times S^3$ is homotopy conjugate
to a standard one, $(t,v)\mapsto (-t,v)$, $(t,v)\mapsto (-t,r(v))$ ($r$ is an hyperplane reflection), $(t,v)\mapsto (t,-v)$,
$(t,v)\mapsto (\overline t,-v)$. The respective orbit spaces $N$ are $S^1\times S^3$, $S^1\tilde\times S^3$, $S^1\times \R P^3$ and $\R P^4 \# \R P^4$.
\begin{theorem}
The index $ind_{\Z_2}(S^1\times S^3, \tau)$ is equal to $1$ if the orbit space is $S^1\times S^3$ or $S^1\tilde\times S^3$, and is equal to $3$ in other cases.
\end{theorem}
 \begin{proof} ~~
 \begin{itemize}
 \item In the case where $N$ is $S^1\times S^3$ or $S^1\tilde\times S^3$, we have $H^2(N,\Z)=\{0\}$ and the Bockstein of the classifying class $\beta(x)$ vanishes. This proves that the $\Z_2$ index is $1$. 
 \item In the case $N=S^1\times \R P^3$, 
the classifying class $x\in H^1( S^1\times \R P^3, \Z_2)$ is equal to $p_2^*(\alpha)$, where $\alpha\in H^3(\R P^3, \Z_2)$ is the generator. Using functoriality and Proposition
\ref{prop5}, we obtain $p^*(\beta^-(x^2))=\beta^-(\alpha^2)\neq 0$. We have 
 $\beta^-(x^2)\neq 0$ and $ind_{\Z_2}(S^1\times S^2, \tau) >2.$ Since $S^1 \times \R P^3$ is an oriented manifold, according to Proposition \ref{prop6}, 
 we have $ind_{\Z_2}(S^1\times S^3, \tau) \neq 4$. Therefore, $ind_{\Z_2}(S^1\times S^3, \tau) =3.$
\item In the case $N=\R P^4 \# \R P^4$, recall that the involution is $(t,v)\mapsto  (\overline{t},-v)$, with $\overline{t}$ the complex conjugate of $t$ in $S^1$. The second projection is equivariant. Functoriality applied to the induced map
$N\rightarrow \R P^3$ again gives $\beta^-x^2\neq 0$ for the classifying class $x\in H^1(N,\Z_2)$. Then we have $ind_{\Z_2}(S^1\times S^3, \tau)\geq 3$, it remains to compute $x^4$.
The cohomology rings $H^*(N,\Z_2)$ will be computed from the modulo $2$ Leray-Serre spectral sequence whose second page is $H^p( \R P^\infty, H^q(S^1 \times S^3, \Z_2))$.  We write below the bottom-left part of the pages $E_3$ and $E_5=E_\infty$.
$$ 
 \begin{tabular}{|c|c|c|c|c|}
\hline 
$\Z_2 $& $\Z_2$  & $\Z_2$ & $\Z_2$ & $\Z_2$ \\ 
\hline 
$\Z_2 $& $\Z_2$  & $\Z_2$ & $\Z_2$ & $\Z_2$ \\
\hline 
0 & 0 & 0 &0& 0 \\ 
\hline 
$\Z_2 $& $\Z_2 $ & $\Z_2$ & $\Z_2$ & $\Z_2$ \\
\hline 
$ \Z_2$ & $\Z_2$ & $\Z_2$& $\Z_2$& $\Z_2$\\ 
\hline 
\end{tabular} \hspace{2cm}
 \begin{tabular}{|c|c|c|c|c|}
\hline 
 0 & 0 & 0 &0& 0\\ 
\hline 
 0 & 0 & 0 &0& 0\\
\hline 
0 & 0 & 0 &0& 0 \\ 
\hline 
$\Z_2 $& $\Z_2 $ & $\Z_2$ & $\Z_2$ & $0$ \\
\hline 
$ \Z_2$ & $\Z_2$ & $\Z_2$& $\Z_2$& $0$\\ 
\hline 
\end{tabular} $$
Hence we have  $H^1(\R P^4 \# \R P^4, \Z_2)\simeq \Z_2 \oplus \Z_2$, and the characteristic class will take the value $(1,0)$ in this case. Also $H^2(\R P^4 \# \R P^4, \Z_2) \simeq \Z_2 \oplus \Z_2$, $H^3(\R P^4 \# \R P^4, \Z_2) \simeq \Z_2 \oplus \Z_2$ and $H^4(\R P^4 \# \R P^4, \Z_2) \simeq \Z_2.$ Since the spectral sequence is compatible with multiplicative structure \cite[Theorem 3]{Serre}, one deduces that $x^2 \neq 0,$ $x^3 \neq 0,$ and $x^4 =0$. Then $ind_{\Z_2}(S^1\times S^3, \tau)\neq 4$.  We conclude that $ind_{\Z_2}(S^1\times S^3, \tau)=3$.  
\end{itemize}
 \end{proof}
\subsection{The secondary obstruction, an example given by surgery presentation}
In this subsection, we  consider an oriented example $(X,\tau)$, $N=X/\tau$, where the primary obstruction $\beta^-(x^2)$ vanishes and compute the secondary obstructions as described in Theorem \ref{th5}.
 
Suppose that $\beta(x)\neq 0$,
the goal will be to compute
\begin{enumerate}
\item the spectral sequence map 
$$\overline d_3^{2,0}:H^2(X,\Z)^{as}\rightarrow \Z_2=H^3(\R P^\infty,\Z^-)\ ,$$
\item the square of any class $y\in H^2(X,\Z)^{as}$ for which $d_3^{0,2}(y)\neq 0$.
\end{enumerate}

We refer to \cite{Gompf} for presentation of closed oriented $4$-dimensional manifolds by framed link diagrams in the sphere $S^3$, with dotted components representing $1$-handles and the other describing  $2$-handles. In the case where the $2$-handlebody obtained from the ball $D^4$ by gluing $1$ and $2$-handles has boundary a connected sum of $S^1\times S^2$, it can be closed with $3$-handles and one $4$-handle giving a closed $4$-manifold uniquely defined up to diffeomorphism.
We consider the closed $4$-manifold $N$ presented in Figure \ref{manifoldN}. 
\begin{figure}[h!]
\begin{center}
\includegraphics[scale=0.4]{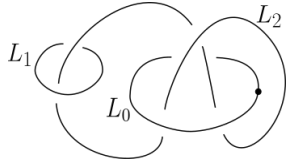} 
\end{center}
\caption{The manifold $N$}
\label{manifoldN}
\end{figure}

An elementary Kirby calculus argument shows that the boundary of the $2$-handlebody obtained with the $1$-handle and the two $2$-handles is $S^1\times S^2$ and we can close with a $3$-handle and a $4$-handle. 
It is useful to understand the attachement sphere for the $3$-handle. The original sphere from the $1$-handle is intersecting the component $L_2$ twice. After attaching the $2$-handles, we get a sphere by removing two small discs around the intersection points, gluing small tubes around $L_2$ and filling with two parallel discs on the boundary of the $2$-handle coming from $L_1$.
The handle structure allows to compute the homology from the complex generated by the cores of the handles. The groups are $C_0=\Z$, $C_1=\Z$, $C_2=\Z\oplus \Z$, $C_3=\Z$, $C_4=\Z$.
After fixing orientations, the boundary maps are $\partial_1=0$, $\partial_2=(0,2)$, $\partial_3=(2,0)$, $\partial_4=0$.
We deduce the homology groups
$H_0(N,\Z)=\Z$, $H_1(N,\Z)=\Z_2$, $H_2(N,\Z)=\Z_2$,  $H_3(N,\Z)=0$, $H_4(N,\Z)=\Z$.
We have $H^1(N,\Z_2)=\Z_2$ which implies that $N$ has a unique non trivial double cover $X$. 
Denote by $\tau: X\rightarrow X$ the deck transformation, then from Lemma \ref{prop6} we have $ind_{\Z_2}(X, \tau)<4.$ 
Also, denote by $\Lambda=\Z[\Z_2]$ the group ring with generator $\tau$. We first compute the homology and cohomology of $N$ with local coefficient $\Lambda$. The handle decomposition of $N$ provides a cell structure for $N$ and an equivariant cell structure for $X$.  Over $\Lambda$ the cell complex for $X$ is free with a generator for each cell of $N$. The choice of the lift in $X$ is fixed by connecting the cell in $N$ to the base point. This allows to read the boundary maps on the Kirby diagram.
The cell chain complex of $N$ with coefficients $\Lambda$ is the following.
$$\Lambda\overset{\tau -1} \longleftarrow \Lambda \overset{(0,1+\tau)} \longleftarrow\Lambda \oplus \Lambda \overset{(1+\tau,0)} \longleftarrow\Lambda \overset{\tau -1} \longleftarrow \Lambda$$
We deduce the homology groups
$H_0(N,\Lambda)=\Lambda/(\tau-1)$, $H_1(N,\Lambda)=0$, $H_2(N,\Lambda)=\Lambda/(1+\tau)\oplus \Lambda^{as}$,  $H_3(N,\Lambda)=0$, $H_4(N,\Lambda)=\Lambda^s$.
Here $\Lambda^s=(1+\tau)\Lambda$ and $\Lambda^{as}=(\tau-1)\Lambda$. The generators for  $H_2(N,\Lambda)$ are $[L_1]=-[\tau L_1]$ and $[(1-\tau)L_2]$.
The dual complex gives the cohomology groups $H^0(N,\Lambda)=\Lambda^s$, $H^1(N,\Lambda)=0$, $H^2(N,\Lambda)= \Lambda^{as}\oplus \Lambda/(1+\tau)$,  $H^3(N,\Lambda)=0$, $H^4(N,\Lambda)=\Lambda/(\tau-1)$.
The generators for  $H^2(N,\Lambda)$ are $[(1-\tau)L_1^\sharp]$ and $[L_2^\sharp]=-[\tau L_2^\sharp]$, where $\sharp$ denote the dual basis of the cochain cell complex.
The cochain complex over $\Z^-$ is 
$$\Z\overset{-2} \longrightarrow \Z \overset{(0,0)} \longrightarrow\Z \oplus \Z \overset{(0,0)} \longrightarrow\Z \overset{-2} \longrightarrow \Z$$
which gives
$H^0(N,\Z^-)=0$, $H^1(N,\Z^-)=\Z_2$, $H^2(N,\Z^-)= \Z\oplus \Z$,  $H^3(N,\Z^-)=0$, $H^4(N,\Z^-)=\Z_2$.
The cohomology of $N$ with coefficient $\Lambda$ coincides with $H^*(X,\Z)$ as a $\Lambda$-module. We get the group $H^2(X,\Z)^{as}=\Z\oplus \Z$.
Observe that the core of the $2$-handle attached on $L_1$ is included in an embedded sphere, while the core of the $2$-handle attached on $L_2$ is included in a projective plane (use the Mobius band bounded by $L_2$). By restricting to these surfaces, functoriality of the exact sequence \eqref{seq_d3} determines the boundary map $\overline d_3^{0,2}: \Z\oplus \Z \rightarrow \Z_2$, $(x,y)\rightarrow \overline y$. 
The secondary obstructions to be computed are the square of the classes $(a,2b+1)\in \Z^2=H^2(X,\Z)^{as}$.
It is convenient to analyze Poincar\'e duality in order to compute the auto-intersection of the corresponding homology classes.
The basis for $H^2(X,\Z)=H^2(X,\Z)^{as}=\Z^2$ is represented by $[L_1^\sharp-\tau L_1^\sharp]$, $[L_2^\sharp]=-[\tau L_2^\sharp]$, while the basis for $H_2(X,\Z)=H_2(X,\Z)^{as}=\Z^2$ is $[L_1]=-\tau[L_1]$ and $[L_2-\tau L_2]$. In the double cover $X$, the first generator is represented by the sphere which lifts the core of the $2$-handle $L_1$ and close with a disk, while the second generator is represented by the band which is the double cover of the Mobius strip closed with the two discs which are the lifts of the core of the $2$-handle $L_2$. The intersection form on generators reads as follows.\\
$$ 
 \begin{tabular}{|c|c|c|}
\hline 
  & $[L_1]$ & $[L_2-\tau L_2] $\\ \hline
 $[L_1]$ & 0 & 1 \\
\hline 
$[L_2-\tau L_2] $ & 1& 0 \\ 
\hline 
\end{tabular} $$
 The Poincaré duals of $[L_1^\sharp-\tau L_1^\sharp]$, $[L_2^\sharp]$ are respectively   $[L_2-\tau L_2] $ and $[L_1]$.
We deduce that the square of $(a,2b+1)\in \Z^2=H^2(X,\Z)=H^2(X,\Z)^{as}$ evaluated on the fundamental class is 
$$ \langle a[L^\sharp_1-\tau L^\sharp_1]+(2b+1)[L^\sharp_2])^2,[N]\rangle= (a[L_2-\tau L_2]+(2b+1)[L_1])^2= 2a(2b+1)\ ,$$
which vanishes for $a=0$.
We obtain the following result.
\begin{proposition}
 The  
$\Z_2$-index of $X$ with the deck involution $\tau$ is $2$.
\end{proposition}

\begin{remark}
It's challenging to work out an example where the obstruction $\beta^-(x^2)$ is zero but the secondary
obstructions never vanish.
\end{remark}

\begin{appendix}
\section*{Appendix: Leray-Serre spectral sequence for twisted cohomology}
Let $(X,\tau)$ be a space with free involution. The Borel construction $X\times_{\Z_2} S^\infty$ is homotopy equivalent to the orbit space $N=X/\tau$. We get a fibration $p: X\times_{\Z_2} S^\infty\rightarrow \R P^\infty$ with fiber $X$. For integer coefficients, the Leray-Serre cohomology spectral sequence of this fibration  has second page
$E_2^{p,q}=H^p(\R P^\infty,H^q(X,\Z))$ and converges to $H^{p+q}(N,\Z)$. Here $H^q(X,\Z)$ is considered as a group with local coefficients. In this appendix we will state a variant of this spectral sequence which computes
the cohomology   $H^*(N,\Z^-)$. Recall that  
the notation $\Z^-$ in the cohomology of the orbit space of an involution means  the local coefficient group $\Z$, where the action of the involution on chains is by $-1$.

Working with CW-complex structures, the Leray-Serre spectral sequence is produced from the filtration of the total space which is the inverse image of the skeleton filtration of the basis. Here we will work with the standard cell decomposition of $\R P^\infty$ which has one cell in each dimension corresponding to an equivariant cell decomposition of $S^\infty$. We also use an equivariant cell decomposition of $X$. We then have a $\Z_2$-equivariant cell decomposition of $X\times S^\infty$. Note that here the involution is $T: (x,y)\mapsto (\tau(x),-y)$.
The cohomology groups $H^*(N,\Z^-)$ are computed from the complex 
$\mathrm{Hom}_{\Z[\Z_2]}(C_*(X\times S^\infty),\Z^-)$,
which is filtrated as usual. We deduce a converging Leray-Serre spectral sequence. We identify below its second page.

\begin{theorem}\label{LSSS}
The Leray-Serre spectral sequence computing $H^*(N,\Z^-)$ from the complex $\mathrm{Hom}_{\Z[\Z_2]}(C_*(X\times S^\infty),\Z^-)$ has second page $$\overline E_2^{p,q}=H^p(\R P^\infty,H^q(X,\Z)\otimes \Z^-)\ .$$
Here 
$H^q(X,\Z)\otimes \Z^-$ is the local group where the generator of $\pi_1(\R P^\infty)$ acts by $-\tau_*$.
\end{theorem}
\begin{proof}[Proof of theorem \ref{LSSS}]~~
Denote by $e_j^q$, $q\geq 0$, $j\in J_q$,  $\tau(e_j^q)=e_{\overline{j}}^q$, the cells of $X$. The standard equivariant cells of $S^\infty$ are $f^p$, $-f^p$, $p\geq 0$.
We have the product cell decomposition of $X\times S^\infty$ by the product cells $e_j^q\times f^p$ and
\mbox{$T(e_j^q\times f^p)=e_{\overline{j}}^q\times (-f^p)$} with filtration degree given by $p$.
The $0$-page of the spectral sequence is 
$$\overline E_0^{p,q}=\mathrm{Hom}_{\Z[\Z_2]}(C_q(X)\otimes (\Z \,f^p\oplus \Z\,(-f^p)),\Z^-)\  .$$
For a cochain $\alpha\in \mathrm{Hom}_{\Z[\Z_2]}(C_q(X)\otimes (\Z \,f^p\oplus \Z\,(-f^p)),\Z^-)$, there is an associated cochain $\dot \alpha \in \mathrm{Hom}(\Z \,f^p\oplus \Z\,(-f^p),\mathrm{Hom}(C_q(X),\Z))$ defined by
$$\dot \alpha(\pm f^p)(e_j^q)=\alpha(e_j^q\otimes (\pm f^p))\ .$$
 It satisfies
 $$\dot \alpha(f^p)(\tau(e_j^q))=\alpha( e_{\overline j}^q\otimes  f^p) =-\alpha( e_{j}^q\otimes (-f^p))=-\dot \alpha(-f^p)(e_j^q)\ .$$
This defines an isomorphism
$$\mathrm{Hom}_{\Z[\Z_2]}(C_q(X)\otimes (\Z \,f^p\oplus \Z\,(-f^p)),\Z^-)\cong
\mathrm{Hom}_{\Z[\Z_2]}((\Z \,f^p\oplus \Z\,(-f^p), \mathrm{Hom}(C_q(X),\Z)\otimes \Z^- ).$$
For the $1$-page, obtained with the vertical differential we get
\mbox{$\overline E_1^{p,q}=\mathrm{Hom}_{\Z[\Z_2]}(C_p(S^\infty) ,H^q(X,\Z)\otimes \Z^-)$}.
The $\overline E_2$ page is then obtained with the horizontal differential, we have
\mbox{$\overline E_2^{p,q}= H^p(\R P^\infty,H^q(X,\Z)\otimes \Z^-)$.}
\end{proof}

\begin{remark}\emph{(Product structure)}\label{product}
The multiplicative structure on the Leray-Serre spectral sequence is constructed in \cite[Theorem 3]{Serre}, see also
\cite[Theorem 5.2]{McCleary}.  This result applies to the pages $E_r^{*,*}$, $r\geq 2$, of the spectral sequence  which computes the cohomologies $H^*(N,\Z)$ and  $H^*(N,\Z_2)$. Here we need a slightly more general version where the product involves both constant coefficient group $\Z$ and twisted coefficient group $\Z^-$. This will be done by considering as coefficient group
$\Lambda=\Z\oplus \Z^-$, with ring structure given by the projection map
$$\mu :\Lambda\otimes \Lambda=\Z\oplus \Z^-\oplus \Z^-\oplus \Z\rightarrow \Z\oplus \Z^-=\Lambda\ ,$$
which means $(a,b)(a',b')=(aa'+bb',ab'+ba')$. 
Here we consider $\Lambda\otimes \Lambda$ as a $\Z_{[\Z_2]}$-module with the diagonal action.
The cochain cell complex with coefficients in $\Lambda$ is a filtrated
differential graded algebra and we obtain a spectral sequence of algebras \cite[Theorem 2.14]{McCleary}. 
The spectral sequence with coefficients in $\Lambda=\Z\oplus \Z^-$ splits as $E_r\oplus \overline E_r$ where $E_r$ is the spectral sequence with coefficients in $\Z$ and $\overline{E_r}$ is the spectral sequence described in Theorem \ref{LSSS}. We get products 
$$E_r^{p,q}\otimes \overline E_r^{p',q'}\rightarrow \overline E_r^{p+p',q+q'}\ \ ,\ \ \overline E_r^{p,q}\otimes \overline E_r^{p',q'}\rightarrow E_r^{p+p',q+q'}\ .$$
Moreover the differentials satisfy the Leibniz rule.
\end{remark}

\end{appendix}
\textbf{Acknowledgements}\\
We are thankful to Pierre Vogel for an enlightening discussion which was very helpful for clarifying the spectral sequence computing homology of a fibration over twisted coefficients exposed in Appendix.

\end{document}